\documentclass{amsart}
\usepackage{amssymb}
\usepackage{amsmath}
\usepackage{amsfonts}

\newtheorem{theorem}{Theorem}
\theoremstyle{plain}

\newtheorem{conclusion}{Conclusion}

\newtheorem{corollary}{Corollary}

\newtheorem{definition}{Definition}
\newtheorem{example}{Example}

\newtheorem{lemma}{Lemma}
\newtheorem{notation}{Notation}

\newtheorem{proposition}{Proposition}

\numberwithin{equation}{section}
\input{tcilatex}

\begin{document}
\title[Semiclassical analysis in B.O. approximation]{Semiclassical analysis
for Hamiltonian in the Born-Oppenheimer approximation}
\author{Senoussaoui Abderrahmane}
\address{Universit\'{e} d'Oran, Facult\'{e} des Sciences Exactes \& Appliqu%
\'{e}es\\
D\'{e}partement de\ Math\'{e}matiques. B.P. 1524 El-Mnaouer, Oran, ALGERIA.}
\email{senoussaoui\_abdou@yahoo.fr, senoussaoui.abderahmane@univ-oran.dz}

\begin{abstract}
The purpose of this paper is to show that the operator 
\begin{equation*}
H\left( h\right) =-h^{2}\Delta _{x}-\Delta _{y}+V\left( x,y\right) ,
\end{equation*}%
$V$ is continuous (or $V\in L^{2}\left( \mathbb{R}_{x}^{n}\times \mathbb{R}%
_{y}^{p}\right) $), and $V\left( x,y\right) \rightarrow \infty $ as $%
\left\Vert x\right\Vert +\left\Vert y\right\Vert \rightarrow \infty ,$ has
purely discrete spectrum. We give an application to the harmonic oscillator.
\end{abstract}

\maketitle

\section{Introduction}

The Born-Oppenheimer approximation is a method introduced in \cite{BoOp} to
analyse the spectrum of molecules. It consists in studying the behavior of
the associate Hamiltonian when the nuclear mass tends to infinity. This
Hamiltonian can be written in the form:%
\begin{equation*}
H\left( h\right) =-h^{2}\Delta _{x}-\Delta _{y}+V\left( x,y\right)
\end{equation*}%
where $x\in \mathbb{R}^{n}$ represents the position of the nuclei, $y\in 
\mathbb{R}^{p}$ is the position of the electrons, $h$ is proportional to the
inverse of the square-root of the nuclear mass and $V\left( x,y\right) $ is
the interaction potential.

In the last decade, many efforts have been made in order to study in the
semiclassical limit the spectrum of $H\left( h\right) $ ( see e.g. \cite{GMS}%
, \cite{KMSW}, \cite{Ma2}, \cite{MaMe}, \cite{MeSe}, \cite{MSD},...). These
authors have shown that in many situations it is still possible to perform,
by Grushin's method, semiclassical constructions related to the existence of
some hidden effective semiclassical operator.

In this paper, we will study the semiclassical approximation to the
eigenvalues and eigenfunctions of $H\left( h\right) $ for potentials $%
V\left( x,y\right) $ with $\inf_{\left\Vert x\right\Vert +\left\Vert
y\right\Vert >R}V\left( x,y\right) ,$ for some $R>0,$ in particular when $%
\lim_{\left\Vert x\right\Vert +\left\Vert y\right\Vert \rightarrow \infty
}V\left( x,y\right) =\infty .$ Our main result in this sens is to show that
in this case the Hamiltonian $H\left( h\right) $ has a purely discrete
spectrum.$\ $The technique used is based on the so called locally compact
operator. The resolvent $R_{H\left( h\right) }\left( z\right) =\left(
H\left( h\right) -z\right) ^{-1}$, $\func{Im}z\neq 0$, of the operator $%
H\left( h\right) $ on $L^{2}\left( \mathbb{R}_{x}^{n}\times \mathbb{R}%
_{y}^{p}\right) $ is typically not compact (however, it usually is on $%
L^{2}\left( \mathbb{X}\right) $, when $\mathbb{X\subset }$ $\mathbb{R}%
_{x}^{n}\times \mathbb{R}_{y}^{p}$ is compact). If $R_{H\left( h\right)
}\left( z\right) $ is compact, then the spectrum $\sigma \left( R_{H\left(
h\right) }\left( z\right) \right) $ is discrete with zero the only possible
point in the essential spectrum. Hence, one would expect that $H\left(
h\right) $ has discrete spectrum with the only possible accumulation point
at infinity (i.e., the essential spectrum $\sigma _{ess}\left( H\left(
h\right) \right) =\emptyset $). In this way, the spectrum $\sigma \left(
H\left( h\right) \right) $ reflects the compactness of $R_{H\left( h\right)
}\left( z\right) $. It turns out that these properties are basically
preserved if, instead of $R_{H\left( h\right) }\left( z\right) $ being
compact, it is compact only when restricted to any compact subset of $%
\mathbb{R}_{x}^{n}\times \mathbb{R}_{y}^{p}$. This is the notion of local
compactness. From an analysis of this notion we will see that the discrete
spectrum of $H\left( h\right) $ is determined by the behavior of $H\left(
h\right) $ on bounded subsets of $\mathbb{R}_{x}^{n}\times \mathbb{R}%
_{y}^{p} $ and the essential spectrum of $H\left( h\right) $ is determined
by the behavior of $V\left( x,y\right) $ in a neighborhood of infinity.

We introduce a specific family of sequences, called Zhislin sequences, which
will allow us to characterize the cress of locally compact, self-adjoint
operators, representing Weyl sequences for a self-adjoint operator.

We finish our work by an application to calculate the spectrum of the
harmonic oscillator of semiclassical Schr\"{o}dinger operator and of the
Hamiltonian in the Born-Oppenheimer approximation $H\left( h\right) .$

\section{Preliminaries}

Let recall some basic definitions on the spectrum of unbounded operator on
Hilbert space.

\begin{definition}
Let $A$ be a linear operator on a Hilbert space $X$ with domain $D\left(
A\right) \subset X$.

\begin{enumerate}
\item The spectrum of $A,$ $\sigma \left( A\right) $, is the set of all
points $\lambda \in \mathbb{C}$ for which $A$ $-\lambda $ ($A$ $-\lambda I,$ 
$I$ is the identity$)$ is not invertible.

\item The resolvent set of $A$, $\rho \left( A\right) $, is the set of all
points $\lambda \in \mathbb{C}$ for which $A$ $-\lambda $ is invertible.

\item If $\lambda \in \rho \left( A\right) $, then the inverse of $A-\lambda 
$ is called the resolvent of $A$ at $\lambda $ and is written as $R_{\lambda
}\left( A\right) =\left( A-\lambda \right) ^{-1}.$
\end{enumerate}
\end{definition}

Let us note that by definition, $\rho \left( A\right) =\mathbb{C}\backslash
\sigma \left( A\right) .$

We can classify $\sigma \left( A\right) $ as:

\begin{definition}
Let $A$ be a linear operator on a Hilbert space $X$ with domain $D\left(
A\right) \subset X$.

\begin{enumerate}
\item If $\lambda \in $ $\sigma \left( A\right) $ is such that $\ker
(A-\lambda )$ $\neq $ $\{0\}$, then $\lambda $ is an eigenvalue of $A$ and
any $u\in $ $\ker (A-\lambda )$, $u\neq 0$, is an eigenvector of $A$ for $%
\lambda $ and satisfies $Au=\lambda u$. Moreover, $\dim \ker (A-\lambda )$
is called the (geometric) multiplicity of $\lambda $ and $\ker (A-\lambda )$
is the (geometric) eigenspace of $A$ at $\lambda $.

\item The discrete spectrum of $A$, $\sigma _{disc}(A)$, is the set of all
eigenvalues of $A$ with finite (algebraic) multiplicity and which are
isolated points of $\sigma (A)$.

\item The essential spectrum of $A$ is defined as the complement of $\sigma
_{disc}(A)$ in $\sigma (A)$: $\sigma _{ess}(A)=\sigma (A)\backslash $ $%
\sigma _{disc}(A)$.
\end{enumerate}
\end{definition}

Let $h\in \left] 0,h_{0}\right] ,$ $h_{0}>0,$ a small semiclassical
parameter.

\begin{theorem}
The spectrum of the self-adjoint operator $-h^{2}\Delta _{x}-$ $\Delta _{y}$
on $H^{2}\left( \mathbb{R}_{x}^{n}\times \mathbb{R}_{y}^{p}\right) $ is%
\begin{equation*}
\sigma \left( -h^{2}\Delta _{x}-\Delta _{y}\right) =\sigma _{ess}\left(
-h^{2}\Delta _{x}-\Delta _{y}\right) =\left[ 0,+\infty \right[ ,\text{ for
all }h\in \left] 0,h_{0}\right] .
\end{equation*}
\end{theorem}

\begin{proof}
The proof is similar as in \cite{HiSi,ReSi}
\end{proof}

Let $V\in L_{loc}^{2}\left( \mathbb{R}_{x}^{n}\times \mathbb{R}%
_{y}^{p}\right) $ and be real. We define $H\left( h\right) =-h^{2}\Delta
_{x}-\Delta _{y}+V\left( x,y\right) $ on $D\left( -h^{2}\Delta _{x}-\Delta
_{y}\right) \cap D\left( V\right) ,$ where $D\left( -h^{2}\Delta _{x}-\Delta
_{y}\right) =$ $H^{2}\left( \mathbb{R}_{x}^{n}\times \mathbb{R}%
_{y}^{p}\right) $ and :%
\begin{equation*}
D\left( V\right) =\left\{ \varphi \in L^{2}\left( \mathbb{R}_{x}^{n}\times 
\mathbb{R}_{y}^{p}\right) ;\text{ }\dint \left\vert V\varphi \right\vert
^{2}dxdy<+\infty \right\} .
\end{equation*}%
Note that $C_{0}^{\infty }\left( 
\mathbb{R}
_{x}^{n}\times 
\mathbb{R}
_{y}^{p}\right) \subset D\left( H\left( h\right) \right) ,$ so $H\left(
h\right) $ is densely defined. The Hamiltonian in the Born-Oppenheimer
approximation is symmetric on this domain:%
\begin{equation*}
\left\langle H\left( h\right) \varphi ,\psi \right\rangle _{L^{2}\left( 
\mathbb{R}_{x}^{n}\times \mathbb{R}_{y}^{p}\right) }=\left\langle \varphi
,H\left( h\right) \psi \right\rangle _{L^{2}\left( \mathbb{R}_{x}^{n}\times 
\mathbb{R}_{y}^{p}\right) },\text{ }\forall \varphi ,\psi \in C_{0}^{\infty
}\left( \mathbb{R}_{x}^{n}\times \mathbb{R}_{y}^{p}\right) ,\text{ }\forall
h\in \left] 0,h_{0}\right] .
\end{equation*}%
Hence, we have that $D\left( H\left( h\right) \right) \subset D\left(
H^{\ast }\left( h\right) \right) .$ Moreover, if $V\geq 0,$ then $H\left(
h\right) \geq 0$ as%
\begin{equation*}
\left\langle H\left( h\right) \varphi ,\varphi \right\rangle _{L^{2}\left( 
\mathbb{R}_{x}^{n}\times \mathbb{R}_{y}^{p}\right) }=\left\Vert h\nabla
_{x}\varphi \right\Vert ^{2}+\left\Vert \nabla _{y}\varphi \right\Vert
^{2}+\left\langle V\varphi ,\varphi \right\rangle _{L^{2}\left( \mathbb{R}%
_{x}^{n}\times \mathbb{R}_{y}^{p}\right) }\geq 0
\end{equation*}%
for any $\varphi \in D\left( H\right) ,$ $\forall h\in \left] 0,h_{0}\right]
.$

\begin{theorem}
Let $V\in L_{loc}^{2}\left( \mathbb{R}_{x}^{n}\times \mathbb{R}%
_{y}^{p}\right) $ et $V\geq 0$. Then the operator $H\left( h\right) $ is
essentially self-adjoint on $C_{0}^{\infty }\left( \mathbb{R}_{x}^{n}\times 
\mathbb{R}_{y}^{p}\right) $, for all $h\in \left] 0,h_{0}\right] $.
\end{theorem}

\begin{proof}
See \cite[Theorem 7.6, page 73 ]{HiSi}, \cite{ReSi}
\end{proof}

\section{Locally compact operators and their application to the
Born-Oppenheimer operator}

\begin{definition}
Let $A$ be a closed operator on $L^{2}(\mathbb{R}^{n})$ with $\rho (A)\neq
\emptyset ,$ let $\chi _{B}$ be the characteristic function for a set $%
B\subset \mathbb{R}^{n}.$ Then $A$ is locally compact if for each bounded
set $B,$ $\chi _{B}\left( A-\lambda \right) ^{-1}$ is compact for some (and
hence all) $\lambda \in \rho (A)$.
\end{definition}

\begin{example}
\label{exemple}

\begin{description}
\item[1.] $\Delta $ is locally compact on $L^{2}(\mathbb{R}^{3})$. Note that 
$\chi _{B}\left( 1-\Delta \right) ^{-1}$ has kernel 
\begin{equation*}
\chi _{B}\left( x\right) \left[ 4\pi \left\Vert x-y\right\Vert \right]
^{-1}e^{-\left\Vert x-y\right\Vert }\text{,}
\end{equation*}%
witch belong to $L^{2}(\mathbb{R}^{3}\times \mathbb{R}^{3}).$ By
Hilbert-Schmidt theorem \cite{HiSi,ReSi,Si}, $\chi _{B}\left( 1-\Delta
\right) ^{-1}$ is compact. We mentionthat the same compactness result holds
in $n$ dimension (see \cite{ReSi}).

\item[2.] $\left( -\Delta \right) ^{\frac{1}{2}},$ the positive square root
of $\left( -\Delta \right) \geq 0$ is locally compact. Indeed, note that it
suffices to show that $A^{\ast }=\chi _{B}\left( i+\left( -\Delta \right) ^{%
\frac{1}{2}}\right) ^{-1}$ is compact. As $A=\left( -i+\left( -\Delta
\right) ^{\frac{1}{2}}\right) ^{-1}\chi _{B},$ we have%
\begin{equation*}
A^{\ast }A=\chi _{B}\left( 1-\Delta \right) ^{-1}\chi _{B},
\end{equation*}%
and by (1) above, $A^{\ast }A$ is compact. Now we claim that this implies
that $A$ is compact, for if $u_{n}\overset{w}{\rightarrow }0$ (weakly
convergence)$,$%
\begin{equation*}
\left\Vert Au_{n}\right\Vert ^{2}=\left\langle u_{n},A^{\ast
}Au_{n}\right\rangle \leq \left\Vert u_{n}\right\Vert \left\Vert A^{\ast
}Au_{n}\right\Vert ,
\end{equation*}%
and as the sequence $\left( u_{n}\right) _{n}$ is uniformly bounded and $%
A^{\ast }Au_{n}\overset{s}{\rightarrow }0$ (strongly convergence)$,$ we have 
$Au_{n}\overset{s}{\rightarrow }0.$ Hence, $A$ is compact.
\end{description}
\end{example}

We now show that certain classes of Hamiltonian operators $H\left( h\right)
=-h^{2}\Delta _{x}-\Delta _{y}+V\left( x,y\right) $ are locally compact.

\begin{theorem}
\label{Theorem1}Let $V$ be continuous (or $V\in L_{loc}^{2}\left( \mathbb{R}%
_{x}^{n}\times \mathbb{R}_{y}^{p}\right) $), $V\geq 0,$ and $V\rightarrow
+\infty $ as $\left\Vert x\right\Vert +\left\Vert y\right\Vert \rightarrow
\infty $ . Then $H\left( h\right) =-h^{2}\Delta _{x}-\Delta _{y}+V\left(
x,y\right) $ is locally compact, for every $h\in \left] 0,h_{0}\right] .$
\end{theorem}

\begin{proof}
Note that $H\left( h\right) $ is self-adjoint by the Kato inequality \cite%
{Ka}, and $H\left( h\right) \geq 0,$ $\forall h\in \left] 0,h_{0}\right] .$
We first make the following claim: 
\begin{equation*}
\left( -h^{2}\Delta _{x}-\Delta _{y}\right) ^{1/2}\text{ is }H^{1/2}\left(
h\right) \text{-bounded and }\left( -h^{2}\Delta _{x}-\Delta _{y}+1\right)
^{1/2}\text{ is }\left( H\left( h\right) +1\right) ^{1/2}\text{-bounded. }
\end{equation*}%
Indeed, since $-h^{2}\Delta _{x}-\Delta _{y}\geq 0$ and $H\left( h\right)
\geq 0,$ all the operators $\left( -h^{2}\Delta _{x}-\Delta _{y}\right)
^{1/2},$ $H^{1/2}\left( h\right) $ and $\left( H\left( h\right) +1\right)
^{1/2}$ are well defined. We have a simple estimate for any $u\in
C_{0}^{\infty }\left( \mathbb{R}_{x}^{n}\times \mathbb{R}_{y}^{p}\right) ,$%
\begin{eqnarray}
\left\Vert \left( -h^{2}\Delta _{x}-\Delta _{y}\right) ^{\frac{1}{2}%
}u\right\Vert ^{2} &=&\left\langle u,\left( -h^{2}\Delta _{x}-\Delta
_{y}\right) u\right\rangle \leq \left\langle u,H\left( h\right)
u\right\rangle \leq \left\langle u,\left( H\left( h\right) +1\right)
u\right\rangle  \notag \\
&\leq &\left\Vert \left( H\left( h\right) +1\right) u\right\Vert ^{2}.
\label{2.1}
\end{eqnarray}%
This estimate extends to all $u\in D\left( H^{1/2}\left( h\right) \right) .$
Consequently, equation $\left( \ref{2.1}\right) $ shows that $\left(
-h^{2}\Delta _{x}-\Delta _{y}\right) ^{\frac{1}{2}}$ is $H^{1/2}\left(
h\right) $-bounded. Also, as we have 
\begin{equation*}
\left\langle u,H\left( h\right) u\right\rangle \leq \left\Vert H^{1/2}\left(
h\right) u\right\Vert ^{2},
\end{equation*}%
which follows from the Schwarz inequality, it follows from this and the
third term of $\left( \ref{2.1}\right) $that $\left( -h^{2}\Delta
_{x}-\Delta _{y}\right) ^{\frac{1}{2}}$ is $H^{1/2}\left( h\right) $
-bounded.

We have%
\begin{gather}
\chi _{B}\left( 1+H\left( h\right) \right) ^{-1/2}=  \notag \\
\chi _{B}\left( 1+H\left( h\right) \right) ^{-1}\left( 1+\left( -h^{2}\Delta
_{x}-\Delta _{y}\right) ^{1/2}\right) ^{-1}\left( 1+\left( -h^{2}\Delta
_{x}-\Delta _{y}\right) ^{1/2}\right) \left( 1+H\left( h\right) \right)
^{-1/2}  \label{2.2}
\end{gather}%
and by example \ref{exemple} (2), the first factor on the right in $\left( %
\ref{2.2}\right) $ is compact, the second is bounded, and so $\chi
_{B}\left( 1+H\right) ^{-1/2}$ is compact. To prove the theorem, simply write%
\begin{equation*}
\chi _{B}\left( 1+H\left( h\right) \right) ^{-1}=\chi _{B}\left( 1+H\left(
h\right) \right) ^{-\frac{1}{2}}\left( 1+H\left( h\right) \right) ^{-\frac{1%
}{2}},
\end{equation*}%
and observe that the right side is product of a compact and a bounded
operator and is hence compact.
\end{proof}

\subsection{Spectral properties of locally compact operators}

We introduce a specific family of sequences, called Zhislin sequences \cite%
{Zh}, which will allow us to characterize the essential spectrum $\sigma
_{ess}$ of locally compact, self-adjoint operators.

\begin{definition}
Let \ $B_{k}=\left\{ x\in \mathbb{R}^{n}:\left\Vert x\right\Vert \leq k,k\in 
\mathbb{N}
\right\} .$ A sequence $\left( u_{n}\right) _{n}$ is a Zhislin for a closed
operator $A$ and $\lambda \in 
\mathbb{C}
$ if $u_{n}\in D(A),$%
\begin{equation*}
\left\Vert u_{n}\right\Vert =1,\mathrm{supp}u_{n}\subset \left\{ x;x\in 
\mathbb{R}^{n}\diagdown B_{n}\right\} \text{ and }\left\Vert \left(
A-\lambda \right) u_{n}\right\Vert \rightarrow 0\text{ as }n\rightarrow
\infty .
\end{equation*}
\end{definition}

By Weyl's criterion \cite{ReSi}, it is clear that if $A$ is self-adjoint and
there exists a Zhislin sequence for $A$ and $\lambda $, then $\lambda \in
\sigma _{ess}\left( A\right) $.

\begin{definition}
Let $A$ be a closed operator. The set of all $\lambda \in \mathbb{C}$ such
that there exists a Zhislin sequence for $A$ and $\lambda $ is called the
Zhislin spectrum of $A$, which we denote by $Z\left( A\right) .$
\end{definition}

\begin{notation}
The commutator of two linear operators $A$ and $B$ is defined formally by $%
\left[ A,B\right] $ $=$ $AB-BA$.
\end{notation}

Let $B(x,R)$ denote the ball of radius $R$ centered at the point $x$. Our
main theorem states that the essential spectrum is equal to the Zhislin
spectrum of a self-adjoint, locally compact operator that is also local in
the sense of $\left( \ref{2.3}\right) $ ahead.

\begin{theorem}
\label{Theorem2}Let $A$ be a self-adjoint and locally compact operator on $%
L^{2}\left( \mathbb{R}^{n}\right) $. Suppose that $A$ also satisfies%
\begin{equation}
\left\Vert \left[ A,\phi _{n}\left( x\right) \right] \left( A-i\right)
^{-1}\right\Vert \rightarrow 0,\text{ as }n\rightarrow \infty  \label{2.3}
\end{equation}%
where $\phi _{n}\left( x\right) =\phi \left( x/n\right) $ for some $\phi \in
C_{0}^{\infty }\left( \mathbb{R}^{n}\right) ,$ supp$\phi \subset $\ $B\left(
0,2\right) ,$ $\phi \geq 0$ and $\phi _{\left\vert B\left( 0,1\right)
\right. }=1$. Then $\sigma _{ess}\left( A\right) =Z\left( A\right) $.
\end{theorem}

\begin{proof}

\begin{enumerate}
\item It is immediate that $Z(A)\subset $ $\sigma _{ess}\left( A\right) $,
by Weyl's criterion. To prove the converse, suppose $\lambda \in $ $\sigma
_{ess}\left( A\right) $. Then there exists a Weyl sequence $\left(
u_{n}\right) _{n}$ for $A$ and $\lambda :$ $\left\Vert u_{n}\right\Vert =1,$ 
$u_{n}\overset{w}{\rightarrow }0$ and $\left\Vert \left( A-\lambda \right)
u_{n}\right\Vert \rightarrow 0.$ Let $\phi _{n}$ be as in the statement of
the theorem, and let $\overline{\phi }_{n}=1-\phi _{n}$. We first observe
that $\left( i-A\right) u_{n}\overset{w}{\rightarrow }0$, because%
\begin{equation}
\left( i-A\right) u_{n}=\left( \lambda -A\right) u_{n}+\left( i-\lambda
\right) u_{n}  \label{2.4}
\end{equation}%
and the first term goes strongly to zero whereas the second goes weakly to
zero. Next, note that by local compactness, for any fixed $n$, $\phi
_{n}u_{m}\overset{w}{\rightarrow }0$ as $m\rightarrow \infty $. This can be
seen by writing%
\begin{equation}
\phi _{n}u_{m}=\phi _{n}\left( i-A\right) ^{-1}\left( i-A\right) u_{m},
\label{2.5}
\end{equation}%
and noting that by $\left( \ref{2.4}\right) $, $\left( i-A\right) u_{m}%
\overset{w}{\rightarrow }0$ and $\phi _{n}\left( i-A\right) ^{-1}$ is
compact. Consequently, $\left\Vert \phi _{n}u_{m}\right\Vert \rightarrow 0$
and $\left\Vert \overline{\phi }_{n}u_{m}\right\Vert \rightarrow 1$ for any
fixed $n$ as $m\rightarrow \infty $.

\item We want to construct a Zhislin sequence from $\overline{\phi }%
_{n}u_{m}.$ To this end, it remains to consider 
\begin{equation}
\left\Vert \left( \lambda -A\right) \overline{\phi }_{n}u_{m}\right\Vert
\leq \left\Vert \overline{\phi }_{n}\right\Vert \left\Vert \left( \lambda
-A\right) u_{m}\right\Vert +\left\Vert \left[ A,\phi _{n}\right]
u_{m}\right\Vert .  \label{2.6}
\end{equation}%
The commutator term is analyzed using $\left( \ref{2.4}\right) $:%
\begin{equation*}
\left\Vert \left[ A,\phi _{n}\right] u_{m}\right\Vert \leq \left\Vert \left[
A,\phi _{n}\right] \left( i-A\right) ^{-1}\right\Vert \left( \left\Vert
\left( \lambda -A\right) u_{m}\right\Vert +\left\vert i-\lambda \right\vert
\right) ,
\end{equation*}%
since $\left\Vert u_{m}\right\Vert =1.$ This converge to zero as $%
n\rightarrow \infty $ uniformly in $m$ because the sequence $\left( \left(
\lambda -A\right) u_{m}\right) _{m}$ is uniformly bounded, say by $M$, so 
\begin{equation*}
\left\Vert \left[ A,\phi _{n}\right] u_{m}\right\Vert \leq \left\Vert \left[
A,\phi _{n}\right] \left( i-A\right) ^{-1}\right\Vert \left( M+\left\vert
i-\lambda \right\vert \right) \rightarrow 0,\text{ as }n\rightarrow \infty 
\text{.}
\end{equation*}

\item To construct the sequence, it follows from $\left( \ref{2.6}\right) $
that for each $k$ there exists $n\left( k\right) $ and $m\left( k\right) $
such that $n\left( k\right) \rightarrow \infty $ and $m\left( k\right)
\rightarrow \infty $ as $k\rightarrow \infty ,$ and 
\begin{equation}
\left\Vert \overline{\phi }_{n\left( k\right) }u_{m\left( k\right)
}\right\Vert \geq 1-k^{-1}  \label{2.7}
\end{equation}%
and 
\begin{equation}
\left\Vert \left( \lambda -A\right) \overline{\phi }_{n\left( k\right)
}u_{m\left( k\right) }\right\Vert \leq k^{-1},  \label{2.8}
\end{equation}%
as $k\rightarrow \infty $. We define $v_{k}=\overline{\phi }_{n\left(
k\right) }u_{m\left( k\right) }\left\Vert \overline{\phi }_{n\left( k\right)
}u_{m\left( k\right) }\right\Vert ^{-1}$. It then follows that $\left(
v_{k}\right) _{k}$ is a Zhislin sequence for $A$ and $\lambda $ by $\left( %
\ref{2.7}\right) $-$\left( \ref{2.8}\right) $ and the fact that supp$v_{k}$ $%
\subset \mathbb{R}^{n}\backslash B_{2k}$. Hence, $\lambda \in Z(A)$ and $%
\sigma _{ess}\left( A\right) \subset Z(A)$.
\end{enumerate}
\end{proof}

We will now apply these ideas to compute $\sigma _{ess}\left( H\left(
h\right) \right) $ of the locally compact Hamlitonian in the
Born-Oppenheimer approximation operators $H\left( h\right) =-h^{2}\Delta
_{x}-\Delta _{y}+V\left( x,y\right) $.

\begin{theorem}
\label{Theorem3}Assume that $V\geq 0$, $V$ is continuous (or $V\in
L^{2}\left( \mathbb{R}_{x}^{n}\times \mathbb{R}_{y}^{p}\right) $), and $%
V\left( x,y\right) \rightarrow \infty $ as $\left\Vert x\right\Vert
+\left\Vert y\right\Vert \rightarrow \infty .$ Then $H=-h^{2}\Delta
_{x}-\Delta _{y}+V\left( x,y\right) $ has purely discrete spectrum.
\end{theorem}

\begin{proof}
By Theorem \ref{Theorem1}, the self-adjoint operator $H$ is locally compact.
Suppose that $h=1,$ and $H\left( 1\right) =H$ for simplification$.$ Let $%
\phi _{q}\left( X\right) $ be as in Theorem \ref{Theorem2}, with $q=n+m$ and 
$X=\left( x,y\right) $. We must verify $\left( \ref{2.3}\right) $. A simple
calculation gives%
\begin{equation}
\left[ H,\phi _{q}\right] =\frac{2}{q}\phi _{q}^{\prime }\nabla _{X}-\frac{1%
}{q^{2}}\phi _{q}^{\prime \prime },  \label{2.9}
\end{equation}%
where $\phi _{q}^{\prime }$ and $\phi _{q}^{\prime \prime }$ are uniformly
bounded in $q$. For any $u\in D\left( H\right) ,$it follows as in $\left( %
\ref{2.1}\right) $ that%
\begin{equation*}
\left\Vert \nabla _{X}u\right\Vert ^{2}\leq \left\langle u,-\nabla
_{X}u\right\rangle \leq \left\langle u,\left( H+1\right) u\right\rangle ,
\end{equation*}%
by the positivity of $V.$ Taking $u=\left( H+1\right) ^{-1}v,$ for any $v\in
L^{2}(\mathbb{R}_{x}^{n}\times \mathbb{R}_{y}^{p}),$ it t follows that $%
\nabla _{X}\left( H+1\right) ^{-1}$ and, consequently, $\nabla _{X}\left(
H-i\right) ^{-1}$ are bounded. This result and $\left( \ref{2.9}\right) $
verify $\left( \ref{2.3}\right) $.

Hence, it follows by Theorem \ref{Theorem2} that $Z\left( H\right) =\sigma
_{ess}\left( H\right) $. We show that $Z\left( H\right) =\left\{ \infty
\right\} $. If $\lambda \in Z\left( H\right) $, then there exists a Zhislin
sequence $\left( u_{q}\right) _{q}$ for $H$ and $\lambda $. By the Schwarz
inequality, we compute a lower bound,%
\begin{eqnarray}
\left\Vert \left( \lambda -H\right) u_{q}\right\Vert &\geq &\left\vert
\left\langle u_{q},\left( \lambda -H\right) u_{q}\right\rangle \right\vert
\geq \left\Vert \nabla _{X}u\right\Vert ^{2}+\left\langle
u_{q},Vu_{q}\right\rangle -\left\vert \lambda \right\vert  \notag \\
&\geq &\left[ \underset{\left( x,y\right) \in \mathbb{R}_{x}^{n}\times 
\mathbb{R}_{y}^{p}\backslash B\left( 0,q\right) }{\inf V\left( x,y\right) }%
\right] -\lambda  \label{2.10}
\end{eqnarray}

As $q\rightarrow \infty $, the left side of $\left( \ref{2.10}\right) $
converges to zero whereas the right side diverges to $+\infty $ unless $%
\lambda =+\infty $. Then $\sigma _{ess}\left( H\right) =\left\{ \infty
\right\} $, that is, is empty.
\end{proof}

\section{Application to the Harmonic Oscillator}

The semiclassical Schr\"{o}dinger operator is $P\left( h\right)
=-h^{2}\Delta +V$, on $L^{2}(\mathbb{R}^{n})$. We treat $h$ as an adjustable
parameter of the theory. We will study the semiclassical approximation to
the eigenvalues and eigenfunctions of $P\left( h\right) =-h^{2}\Delta +V$
for potentials $V$ with in particular when $\lim_{\left\Vert x\right\Vert
\rightarrow \infty }V\left( x\right) =\infty $. Because the small parameter $%
h$ appears in front of the differential operator $-\Delta $, it may not be
clear what is happening as $h$ is taken to be small. It is more convenient,
and perhaps more illuminating, to change the scaling. Letting $\lambda =1/h$%
, we rewrite the Schr\"{o}dinger operator as%
\begin{equation*}
P\left( \lambda \right) =-\Delta +\lambda ^{2}V=h^{-2}P\left( h\right) ,
\end{equation*}%
looking at $P\left( \lambda \right) $, we see that the semiclassical
approximation involves, $\lambda \rightarrow \infty $.

\begin{definition}
Let $A$ a real $n\times n$ matrix, $A$ is a positive definite matrix if $%
\left\langle Ax.x\right\rangle _{\mathbb{R}^{n}}>0$, for all $x\in \mathbb{R}%
^{n}$.
\end{definition}

\begin{definition}
Let $A$ a symmetric, positive definite matrix. The Schr\"{o}dinger operator
of type:%
\begin{equation}
K\left( \lambda \right) =-\Delta +\lambda ^{2}\left\langle Ax,x\right\rangle
_{\mathbb{R}^{n}}  \label{3.1}
\end{equation}%
is said to be the harmonic oscillator.
\end{definition}

Here $\left\langle x,Ax\right\rangle _{\mathbb{R}^{n}}=\dsum%
\limits_{i,j=1}^{n}a_{ij}x_{i}x_{j}$ is the Euclidean quadratic form is
bounded from below by%
\begin{equation*}
\left\langle Ax,x\right\rangle _{\mathbb{R}^{n}}\geq \lambda _{\min
}\left\Vert x\right\Vert ^{2},
\end{equation*}%
where $\lambda _{\min }$is the smallest eigenvalue of $A$ and is strictly
positive.

We see that $K\left( \lambda \right) $ is positive with a lower bound
strictly greater than zero. Since the harmonic oscillator is continuous and $%
V_{har}\left( x\right) =\left\langle x,Ax\right\rangle _{%
\mathbb{R}
^{n}}\rightarrow \infty $, as $\left\Vert x\right\Vert \rightarrow \infty $,
the harmonic oscillator Hamiltonian $\left( \ref{3.1}\right) $ is
self-adjoint. Moreover, the spectrum of $K\left( \lambda \right) $, $\sigma
\left( K\left( \lambda \right) \right) $, is purely discrete by Theorem \ref%
{Theorem3}.

We would like to find out how the eigenvalues of $K\left( \lambda \right) $
depend on the parameter $\lambda $.

\begin{definition}
Two operators $A$ and $B$, with $D(A)=D(B)=D$, are called similar if there
exits a bounded, invertible operator $C$ such that $CD\subset D$ and $%
A=CBC^{-1}.$
\end{definition}

\begin{proposition}
\label{Proposition1}If $A$ and $B$ are similar, then $\sigma (A)=\sigma (B)$.
\end{proposition}

\begin{proof}
It suffices to show that 
\begin{equation*}
\mu \in \rho \left( A\right) \Longleftrightarrow \mu \in \rho \left( A\right)
\end{equation*}%
where $\rho \left( A\right) :=\mathbb{C}\backslash \sigma (A)$ is the
resolvent set. This comes from 
\begin{equation*}
A-\lambda I=C\left( B-\lambda I\right) C^{-1}.
\end{equation*}
\end{proof}

\begin{definition}
\bigskip For $\theta \in \mathbb{R}_{+},$ we define, the so-called dilation
group, is a map on any $\psi \in C_{0}^{\infty }\left( \mathbb{R}^{n}\right) 
$ by 
\begin{equation*}
U_{\theta }\psi \left( x\right) =\theta ^{n/2}\psi \left( \theta x\right) .
\end{equation*}
\end{definition}

\begin{lemma}
The dilation $U_{\theta }$ is an unitary on $L^{2}\left( \mathbb{R}%
^{n}\right) \rightarrow $ $L^{2}\left( \mathbb{R}^{n}\right) $, and 
\begin{equation*}
\left( U_{\theta }\right) ^{\ast }=\left( U_{\theta }\right) ^{-1}=U_{\theta
^{-1}}.
\end{equation*}%
We have also, for $\theta ,\theta ^{\prime }\in \mathbb{R}_{+},$%
\begin{equation*}
U_{\theta }U_{\theta ^{\prime }}=U_{\theta +\theta ^{\prime }}.
\end{equation*}
\end{lemma}

\begin{proof}
Easy proof.
\end{proof}

We now claim that $U_{\lambda ^{-\frac{1}{2}}}$ implements a similarity
transformation on $K\left( \lambda \right) $ by%
\begin{equation}
U_{\lambda ^{-\frac{1}{2}}}K\left( \lambda \right) U_{\lambda ^{-\frac{1}{2}%
}}^{-1}=\lambda K  \label{3.2}
\end{equation}%
where%
\begin{equation}
K=-\Delta +\left\langle Ax,x\right\rangle _{\mathbb{R}^{n}}  \label{3.3}
\end{equation}%
Now we compute the spectrum of the harmonic oscillator $K$.

\begin{proposition}
The eigenvalues of $K$ are given by 
\begin{equation*}
\sigma \left( K\right) =\left\{ \dsum\limits_{i=1}^{n}\left( 2n_{i}+1\right)
w_{i};\text{ }n_{i}\in \mathbb{Z}_{+}\cup \left\{ 0\right\} \right\} ,
\end{equation*}%
where $\left\{ w_{i}^{2}\right\} _{i=1}^{n}$ are the eigenvalues of the
matrix $A$.
\end{proposition}

\begin{proof}
The proof is by induction on the dimension $n\in \mathbb{N}^{\ast }.$

For $n=1,$ the Hermite polynomials $\mathcal{H}_{p}$ are defined by%
\begin{equation*}
\left( \frac{d}{dx}\right) ^{2}\left( e^{-x^{2}}\right) =\left( -1\right)
^{p}\mathcal{H}_{p}e^{-x^{2}}.
\end{equation*}%
We recall that they satisfy the relations%
\begin{equation*}
\mathcal{H}_{0}=1\text{ et }\mathcal{H}_{p+1}=\left( -\frac{d}{dx}+2x\right) 
\mathcal{H}_{p},p\geq 0.
\end{equation*}%
Hermite functions $\Psi _{n}$ are defined by%
\begin{equation*}
\Psi _{p}=C_{p}\mathcal{H}_{p}e^{-x^{2}/2}\text{, where }C_{p}=\left( \sqrt{%
\pi }2^{p}p!\right) ^{-1/2}.
\end{equation*}%
For $p\geq 1,$ we have%
\begin{equation}
\left( \frac{d}{dx}+x\right) \Psi _{0}=0  \label{3.4}
\end{equation}%
and 
\begin{equation}
\left( -\frac{d}{dx}+x\right) \Psi _{p}=\sqrt{2\left( p+1\right) }\Psi
_{p+1}.  \label{3.5}
\end{equation}%
Now, if $H$ is the harmonic oscillator in one dimension%
\begin{equation*}
H=-\left( \frac{d}{dx}\right) ^{2}+x^{2},
\end{equation*}%
it follows from $\left( \ref{3.4}\right) $ and $\left( \ref{3.5}\right) $
that%
\begin{equation*}
H\Psi _{p}=\left( 2p+1\right) \Psi _{p}.
\end{equation*}
\end{proof}

\begin{corollary}
As a consequence of Proposition \ref{Proposition1} and $\left( \ref{3.2}%
\right) $, $\sigma \left( K\left( \lambda \right) \right) =\lambda \sigma
\left( K\right) $, where $\sigma \left( K\right) $ is independent of $%
\lambda $. Hence the eigenvalues of $K\left( \lambda \right) $ depend
linearly on $\lambda $. Moreover, the multiplicities of the related
eigenvalues are the same. Now 
\begin{equation*}
\sigma \left( K\left( \lambda \right) \right) =\left\{
\dsum\limits_{i=1}^{n}\left( 2n_{i}+1\right) \lambda w_{i};\text{ }n_{i}\in 
\mathbb{Z}_{+}\cup \left\{ 0\right\} \right\}
\end{equation*}%
where $\left\{ w_{i}^{2}\right\} _{i=1}^{n}$ are the eigenvalues of the
matrix $A.$ The eigenfunctions are related through\bigskip\ the unitary
operator $U_{\lambda ^{-\frac{1}{2}}}.$ If $\Psi _{p}$ are the
eigenfunctions of $K,$ then $\widetilde{\Psi _{p}}=U_{\lambda ^{-\frac{1}{2}%
}}\Psi _{p}$ $\ $are the eigenfunctions of $K\left( \lambda \right) .$
\end{corollary}

\begin{conclusion}
The semiclassical harmonic oscillator $P\left( h\right) =-h^{2}\Delta
+\left\langle Ax,x\right\rangle _{\mathbb{R}^{n}}$ has purely discrete
spectrum 
\begin{equation*}
\sigma \left( P\left( h\right) \right) =\left\{ he_{j},j\in \mathbb{Z}%
_{+}\cup \left\{ 0\right\} \right\}
\end{equation*}%
where $e_{j}\in \sigma \left( K\right) .$
\end{conclusion}

In general, we can give the spectrum the harmonic oscillator in the
Born-Oppenheimer Approximation 
\begin{equation*}
H\left( h\right) =-h^{2}\Delta _{x}-\Delta _{y}+\left\langle
Ax,x\right\rangle _{\mathbb{R}_{x}^{n}}+\left\langle By,y\right\rangle _{%
\mathbb{R}_{x}^{p}}\text{ on }L^{2}\left( \mathbb{R}_{x}^{n}\times \mathbb{R}%
_{y}^{p}\right)
\end{equation*}%
where $A$ and $B$ are two symmetric, positive definite matrix.%
\begin{equation*}
\sigma \left( H\left( h\right) \right) =\sigma _{disc}\left( H\left(
h\right) \right) =\left\{ \dsum\limits_{i=1}^{n}\left( 2n_{i}+1\right)
hw_{i}+\dsum\limits_{i=1}^{p}\left( 2n_{i}+1\right) \mu _{i},n_{i}\in 
\mathbb{Z}_{+}\cup \left\{ 0\right\} \right\}
\end{equation*}%
where $\left\{ w_{i}^{2}\right\} _{i=1}^{n}$ and $\left\{ \mu
_{i}^{2}\right\} _{i=1}^{p}$ are respectively the eigenvalues of the matrix $%
A$ and $B.$

\end{document}